\documentclass[smallheadings,headsepline,11pt]{amsart}
\usepackage[english]{babel}
\usepackage[all]{xy}
\usepackage{amssymb, amsmath}
\usepackage{txfonts}
\usepackage{enumitem}
\usepackage{calrsfs}
\usepackage{amscd}
\usepackage{graphicx}
\usepackage[applemac]{inputenc}
\usepackage[T1]{fontenc} 
\usepackage{hyperref} \usepackage{amsfonts}
\usepackage{cite}
\usepackage{environ}
\usepackage{amscd}
\usepackage{xcolor}

\usepackage{amsthm}

\DeclareMathOperator*{\Gr}{Gr}

\newcommand{\Ba}{\mathrm{B}}
\newcommand{\Lo}{\mathrm{L}}
\newcommand{\Om}{\Omega}
\newcommand{\Ne}{\mathrm{N}}
\newcommand{\Ch}{\mathrm{C}}
\newcommand{\Mo}{\mathrm{M}}
\newcommand{\G}{\mathrm{G}}

\newcommand{\sSet}{\mathbf{sSet}}
\newcommand{\dgco}{\mathbf{dgCoa}^{\mathrm{conil}}}
\newcommand{\dga}{\mathbf{dgA}_{/k}}
\newcommand{\qCat}{\mathbf{qCat}}
\newcommand{\Mon}{\mathbf{Mon}}
\newcommand{\sMon}{\mathbf{sMon}}

\DeclareMathOperator{\Hom}{Hom}

\DeclareMathOperator{\Fun}{Fun}

\DeclareMathOperator{\De}{D}
\DeclareMathOperator{\Dco}{D^{co}}

\input diagxy

\newtheorem{theorem}{Theorem}[section] 
\newtheorem{corollary}[theorem]{Corollary}
\newtheorem{lemma}[theorem]{Lemma}
\newtheorem{proposition}[theorem]{Proposition}
\theoremstyle{remark}
\newtheorem{remark}[theorem]{Remark}
\newtheorem{example}[theorem]{Example}

\begin{document}
\markright{Maurer-Cartan moduli}
\bibliographystyle{../hsiam2}

\setcounter{tocdepth}{2}
\setlength{\parindent} {0pt}
\setlength{\parskip}{1ex plus 0.5ex}

\newcommand{\noproof}{\hfill \ensuremath{\Box}}

\newcommand{\cat}[1]{\mathcal{#1}} 
\newcommand{\ob}{\textrm{ob }}
\newcommand{\mor}{\textrm{mor }}
\newcommand{\id}{\mathbf 1} 
\newcommand{\mods}{\textrm{-Mod}}

\newcommand{\watchit}{\marginpar{$\bigstar$}}

\newcommand{\ground}{k}

\newcommand{\set}[1]{\mathbb{#1}}
\newcommand{\Q}{\mathbb{Q}}
\newcommand{\C}{\mathbb{C}}
\newcommand{\Z}{\mathbb{Z}}
\newcommand{\R}{\mathbb{R}}

\newcommand{\A}{\mathcal{A}}
\newcommand{\B}{\mathcal{B}}
\newcommand{\K}{\mathcal{K}}

\newcommand{\Ga}{\Gamma}
\newcommand{\eps}{\epsilon}
\newcommand{\de}{\delta}
\newcommand{\la}{\lambda}
\newcommand{\al}{\alpha}
\newcommand{\om}{\omega}

\renewcommand{\to}{\rightarrow}
\newcommand{\oo}{\infty}
\newcommand{\di}{\mbox{d}}

\newcommand{\op}{^{\textrm{op}}} 

\newcommand{\comment}[1]{}
\newcommand{\margin}[1]{\marginpar{\footnotesize #1}}

\NewEnviron{killcontents}{}

\title{Homotopy theory of monoids and derived localization}
\author{Joe Chuang, Julian Holstein, Andrey Lazarev} \thanks{This work was partially supported by EPSRC grants EP/N015452/1 and EP/N016505/1}

\date{}
\maketitle
\begin{abstract}
We use derived localization of the bar and nerve constructions to provide simple proofs of a number of results in algebraic topology, both known and new. This includes a recent generalization of Adams's cobar-construction to the non-simply connected case, and a new algebraic model for the homotopy theory of connected topological spaces as an $\infty$-category of discrete monoids.
\end{abstract}
\setcounter{tocdepth}{1}
\tableofcontents

\section{Introduction}

In this paper we examine the consequences of the close relationship between the topological classifying space construction and the algebraic bar construction combined with the techniques of derived localization of differential graded (dg) algebras.

Let $M$ be any discrete monoid with a subset $W$. We consider its monoid algebra $\Ch (M)$ and its derived localization $\Lo_{W}\Ch (M)$; it is a dg algebra obtained from $\Ch (M)$ by inverting the elements in $W$ in a homotopy invariant fashion, \cite{BCL18}. Let $\Ba\Lo_{W}\Ch (M)$ be the bar construction on $\Lo_{W}\Ch(M)$. 
On the other hand, let $\Ne (M)$ be the nerve (classifying space) of $M$ considered as an $\infty$-category and $\Lo_{W}\Ne(M)$ be its localization at $W$, viewed as 1-morphisms in $\Ne(M)$.
Finally let $\Ch \Lo_{W}\Ne(M)$ be the normalized chain coalgebra of the simplicial set $\Lo_{W}\Ne(M)$.
Then we prove that the dg coalgebras $\Ba \Lo_{M}\Ch(M)$ and $\Ch \Lo_{W}\Ne(M)$ are weakly equivalent, i.e.\ there is a zig-zag of filtered quasi-isomorphisms between them.

We can then deduce the following results with minimal computation:
\begin{enumerate}
\item For any reduced grouplike (in particular Kan or 1-reduced) simplicial set $K$ there is an equivalence between $\Ch\G (K)$, the chain algebra of the loop group of $K$, and $\Om\Ch(K)$, the cobar construction on the chain coalgebra of $K$. See Corollary \ref{cor-loopcobar}. This generalizes a classical result of Adams \cite{Adams56}.
\item For an arbitrary reduced simplicial set $K$ there is an equivalence between $\Ch \G (K)$ and a localization of $\Om \Ch (K)$. See Corollary \ref{cor-hesstonks}.
\item The derived category of second kind of the chain coalgebra on a reduced simplicial set $K$ contains the derived category of $\infty$-local systems on $|K|$. If $K$ is grouplike the categories are equivalent. See Corollary \ref{cor-derivedcats}.
\item Two reduced Kan complexes are weakly equivalent if and only if there is a weak equivalence between their integer-valued chain coalgebras. See Corollary \ref{cor-detectwes}.
\end{enumerate}


Some of these, or similar, results have appeared in the literature before: (1) was shown when $K$ is a simplicial singular set of a topological space by Rivera-Zeinalian in \cite{Zeinalia16}, (2) is equivalent to the extended cobar construction of Hess-Tonks \cite{Hess10a}, and (4) is originally due to Rivera-Zeinalian \cite{Rivera18}. However, we believe this paper significantly simplifies the existing proofs and adds conceptual clarity. In particular, we show that the extended cobar-construction of Hess and Tonks \cite{Hess10a} of the chain coalgebra of a simplicial set is a derived localization of the ordinary cobar-construction and clarify its dependence on the choices made. 

The main theorem of this paper is a new result, which provides an entirely algebraic model for the homotopy category of connected spaces. 
By inverting those maps of discrete monoids which induce quasi-isomorphisms of derived localized monoid algebras one obtains an $\infty$-category of discrete monoids. More precisely, this $\infty$-category is realized as a \emph{relative category} in the sense of Barwick and Kan \cite{Barwick12}.
We prove in Theorem \ref{thm-main} that this $\infty$-category of discrete monoids is equivalent to the $\infty$-category of reduced simplicial sets (also viewed as a relative category with ordinary weak equivalences of simplicial sets). This is potentially of great computational utility since derived localizations of associative rings are effectively computable in a number of situations, both of algebraic and topological origin cf. \cite{BCL18}.

As far as we know, this is the first result providing an algebraization of the homotopy category of spaces without any restrictions apart from connectivity (such as simple connectivity, rationality or being of finite type). It is ideologically similar to the well-known result of Thomason \cite{Tho80} constructing a closed model category structure on small categories that also models the $\infty$-category of spaces as well as its refinement due to Raptis \cite{Rap10}. However Thomason's and Raptis's constructions (while providing more structured equivalences of closed model categories) cannot be viewed as genuine algebraization results since weak equivalences of small categories are defined by appealing to the category of spaces.

\subsection{Notation}
We work over a commutative ground ring $k$ that is a principal ideal domain.   All tensor products are understood over $k$.

We denote the category of simplicial sets by $\mathbf{sSet}$ and its subcategory of reduced simplicial sets, i.e.\ simplicial sets with exactly on 0-simplex, by $\mathbf{sSet}^{0}$. 
We write $\mathbf{qCat}$ for the category of simplicial sets with the Joyal model structure as a model for $\infty$-categories; the subcategory of simplicial sets with one object is denoted by $\mathbf{qCat}^{0}$. 
To distinguish the classical weak equivalences in $\mathbf{sSet}$ and the categorical equivalences in $\mathbf{qCat}$ we will denote them by $\simeq_{Q}$ (for Quillen) and $\simeq_{J}$ (for Joyal) respectively. The geometric realization of a simplicial set $K$ will be denoted by $|K|$.

We denote the category of monoids by $\Mon$ and that of simplicial monoids by $\sMon$. 

The category of unital dg algebras, free as $k$-modules, is denoted by $\mathbf{dgA}$ and the category of augmented dg-algebras by $\mathbf{dgA}_{/k}$.
We denote by $\dgco$ the dg category of counital conilpotent dg coalgebras, also free as $k$-modules. 
By weak equivalences of dg coalgebras we always mean morphisms in the class generated by filtered quasi-isomorphism, the definition is recalled in Section \ref{sect-barcobar}. All our gradings are homological.

We will denote by $\Ch$ the normalized chain coalgebra functor with coefficients in $k$ on $\sSet$, cf.\ Chapter 10 of \cite{Neisendorfer10}. We also denote by $\Ch$ the functor that sends any monoid to its monoid algebra over $k$, it will be viewed as an object of $\mathbf{dgA}$. 

\subsection{Acknowledgements}
The authors would like to thank Michael Batanin, Jonathan Block, Denis Cisinski, Kathryn Hess and Andy Tonks for useful input.

\section{Background}

\subsection{The bar cobar adjunction}\label{sect-barcobar}
We recall that over any commutative ring the bar and cobar construction provide an adjunction $\Omega: \dgco \leftrightarrows \mathbf{dgA}_{/k}: \Ba$. See for example \cite{Husemoller74}.

For the reader's convenience we repeat some definitions.
For an augmented dg algebra $\epsilon: A \to k$, set $A_{+}=\ker(\epsilon)$. Then define $\Ba(A) = \oplus_{n=0}^{\infty}  (sA_{+})^{\otimes n}$ with comultiplication defined by deconcatenation and counit given by the projection to $(sA_{+})^{\otimes 0}\cong k$, where $s$ denotes the suspension. 
We define the differential on $\Ba(A)$ to be the unique coderivation whose projection $\Ba(A) \to sA_+$ restricts to $d_{sA}$ on $sA_+$, to
$s \mu_A (s^{-1}\otimes s^{-1})$ on $sA_+\otimes sA_+ $ and to $0$ on higher tensors.
The cobar construction of a coalgebra is defined analogously.

Now assume that $k$ is a field.
Then the bar-cobar adjunction is a Quillen equivalence \cite{Positselski11}. We consider the usual model structure on augmented dg algebras (so that weak equivalences are multiplicative quasi-isomorphisms).
For the model structure on $\dgco$ see  \cite[Theorem 9.3(b)]{Positselski11}. The key definition is that $f: C \to D$ is a filtered quasi-isomorphism if there are admissible filtrations on $C$ and $D$ such that the associated graded map $\Gr(f)$ is a graded quasi-isomorphism. A filtration $F$ on a conilpotent coalgebra $C$ is admissible if it is increasing, compatible with comultiplication and differential, and $F^{0}$ equals the image of the coaugmentation $k \to C$. An admissible filtration always exists. Then $f: C \to D$ is a weak equivalence in $\dgco$ if it is contained in the smallest class of morphisms containing filtered quasi-isomorphisms and closed under the 2-out-of-3 property. If $k$ is not a field we will, somewhat abusing terminology, still refer to filtered quasi-isomorphisms as weak equivalences, even though there may not be an underlying closed model category.
Cofibrations in $\dgco$ are just monomorphisms.

\subsection{Localization of dg algebras}\label{sect-alglocal}
Given a dg algebra $A$ with a collection of cycles $S$, its derived localization $\Lo_{S}A$ is the homotopy initial dg algebra under $A$ such that the images of all $s \in S$ are invertible in homology, \cite[Definition 3.3]{BCL18}.
By  \cite[Theorem 3.10]{BCL18}, $\Lo_{S}(A)$ is a homotopy pushout of the form $A *^{h}_{k\langle S \rangle} k\langle S, S^{-1}\rangle$. 

\subsection{Localization of $\infty$-categories}\label{sect-catlocal}

We will use Joyal's theory of $\infty$-categories as quasi-categories, see \cite{Joyal?, Lurie11a} for further background. Given any simplicial set $K$ with a subsimplicial set $W$ we may consider it as an object of $\qCat$ and define its localization $\Lo_{W}K$, see \cite[Proposition 7.1.3]{Cisinski?}. It has the universal property that for any quasi-category $C$ the functor category $\Fun(\Lo_{W}K, C)$ is equivalent to the subcategory of $\Fun(K, C)$ consisting of functors sending any map in $W$ to an invertible map in $C$. See also the section on homotopy localization in \cite{Joyal?}.

We restrict attention to reduced simplicial sets. We are particularly interested in the case where $W$ is given by a collection of 1-simplices $S$ and will write $\Lo_{S}K$ in this case.
Let $I$ be the nerve of $\set N$, the free monoid on one generator, and $J$ the nerve of $\set Z$, the free group on one generator.
There are natural maps $I \to J$ and $\amalg_{S}I \to K$, and $\Lo_{S}K$ is equivalent to the homotopy pushout in $\qCat^{0}$ of $\amalg_{S}J \leftarrow \amalg_{S}I \to K$. This follows from the proof of  \cite[Proposition 7.1.3]{Cisinski?}: The map $\amalg_{S}I \to \amalg_{S}J$ is an anodyne extension, i.e.\ a trivial cofibration in the Quillen model structure, thus it may play the role of $W \to W'$ and the rest of the proof applies without changes.

\subsection{Grouplike simplicial sets}\label{sect-grouplike}
Any simplicial set $K$ may be interpreted as an object in $\qCat$ and its fundamental category $\pi(K)$ is defined as the left adjoint of the nerve functor from categories to simplicial sets.
If $K$ is weakly Kan, there is an explicit construction of $\pi(K)$ as the category with objects given by 0-simplices and morphisms given by 1-simplices modulo 2-simplices, see \cite[Section 1.2.3]{Lurie11a}.

We say $K$ is grouplike if $\pi(K)$ is a groupoid.
In particular all Kan complexes are grouplike. A converse is true for weak Kan complexes:
if $K$ is a weak Kan complex and grouplike then it is a Kan complex, see \cite[Proposition 1.2.5.1]{Lurie11a}.
The property of being grouplike is invariant under categorical equivalence, thus the Joyal fibrant replacement of a grouplike simplicial set is a Kan complex.

\subsection{Relative categories}
We will also use the theory of relative categories as introduced in \cite{Barwick12} as a model for $\infty$-categories. A relative category $(\cat C, W)$ is just a pair consisting of a category $\cat C$ and a class of weak equivalences $W \subset \mathrm{Mor}(\cat C)$.

Associated to any relative category $(\cat C, W)$ is a simplicial category $\Lo_{W}\cat C$ obtained by simplicial localization of $\cat C$ (viewed as a simplicial category) at $W$.  There is a model structure on relative categories whose weak equivalences $(\cat C, W)\to (\cat C^{\prime}, W^{\prime})$ are those maps that induce weak equivalences of simplicial localizations 
$\Lo_{W}\cat C\to \Lo_{W^{\prime}}\cat C^{\prime}$, cf. \cite{BaK12}.

The model category of relative categories is Quillen equivalent to the model categories of simplicial categories and quasi-categories. In particular the relative category $(\sSet, W_{Q})$, where $W_{Q}$ denotes weak homotopy equivalences, is a model for the $\infty$-category of spaces.

We are not aware of a good exposition of homotopy limits and colimits in relative categories. To avoid technicalities we define the homotopy limit of a diagram in a relative category by taking the $\infty$-categorical limit of the corresponding diagram in the associated $\infty$-category. A comparison result ensures that if the relative category happens to be a model category then this recovers the usual homotopy limits and homotopy colimits. This is explained in Remark 7.9.10 of \cite{Cisinski?} or Remark 2.5.8 in \cite{Barnea17}.
In particular it follows from this that any weak equivalence of relative categories preserves  homotopy limits, which we will need  below.

\section{Bar and nerve construction}
We begin by considering the following diagram.

\[
\begin{CD}
\Mon     @>\Ne >>  \qCat^{0}\\
@VV\Ch V        @VV\Ch V\\
\dga     @>\Ba >>  \dgco
\end{CD}
\]
Here $\Ne$ is the usual nerve of a monoid, considered as a reduced simplicial set.
The vertical arrows are given, respectively, by the monoid algebra and the normalized chain coalgebra, over $k$. For any monoid $M$ the augmentation $\epsilon$ on $\Ch (M)$ is induced by $M \to *$.
Finally, $\Ba$ is the bar construction on an augmented dg algebra as recalled in Section \ref{sect-barcobar}.

It is a straightforward but fundamental observation that this diagram commutes:
\begin{lemma}\label{lemma-discretecommute}
For any monoid $M$ there is a natural isomorphism of dg coalgebras $\Ch \Ne M \cong \Ba \Ch M$.
\end{lemma}
\begin{proof}
With the augmentation $\epsilon: \Ch(M) \to k$ given as above we write $\underline M = \ker \epsilon$ and $\underline m$ for $m-1$ in $\Ch M$. Write $\overline M$ for $M\setminus \{1\}$. Then the set of elements $\underline m$ for $m \in \overline{M}$ gives a basis for $\underline M$. The map $m \mapsto s \underline m$ induces an isomorphism from $\Ch_{n}\Ne M = k[\overline M^{\times n}]$ to $\Ba \Ch M_{n} = (s\underline M)^{\otimes n}$.
A quick computation shows that the differentials also match, as do the coalgebra structures.
\end{proof}

We will refine this result by considering localizations of dg algebras and simplicial sets.

\begin{lemma}
There is a natural model structure on $\mathbf{qCat}^{0}$ such that weak equivalences are categorical equivalences and cofibrations are monomorphisms. 
\end{lemma}
\begin{proof}
We recall the Quillen equivalence  $\mathfrak C \dashv \cat N: \qCat \leftrightarrows \mathbf{sCat}$, see e.g.\ \cite{Lurie11a} and observe that it restricts to an adjunction $\mathbf{sMon} \leftrightarrows \qCat^{0}$.
Then the proof of the lemma is the same as for the non-reduced case, cf. \cite[Theorem 2.2.5.1]{Lurie11a}.
We need to check three conditions: 
\begin{enumerate}
\item 
The class of categorical equivalences in $\qCat^{0}$ is perfect in the sense of  \cite[Definition A.2.6.10]{Lurie11a}. Namely, it contains isomorphisms, is closed under 2-out-of-3, is stable under filtered colimits and is generated under filtered colimits by a small subset.
As the class of weak equivalences in $\sMon$ are perfect it suffices to check that $\mathfrak C$ preserves filtered colimits by \cite[Corollary A2.6.12]{Lurie11a}. 
But $\mathfrak C: \qCat^{0} \to \mathbf{sMon}$ commutes with colimits. 
\item Categorical equivalences are stable under pushout by cofibrations. Cofibrations in $\qCat^{0}$ are also cofibrations in $\qCat$, so this follows from the non-reduced case (or directly by the same argument).
\item Finally we need to check that a map $f: K \to L$ of reduced simplicial sets which has the right lifting property with respect to all cofibrations is a categorical equivalence.
It suffices to show that if $f$ has the right lifting property with respect to all cofibrations between reduced simplicial sets then it has the right lifting property with respect to all cofibrations; this reduces the problem to the non-reduced case.
So let $A \to B$ be a cofibration. But any maps $A \to K$ and $B \to L$ factor through the reduced simplicial sets $\bar A = A/A_{0}$ and $\bar B = B/B_{0}$, and $\bar A \to \bar B$ is a cofibration. 
Thus the right lifting property with respect to $\bar A \to \bar B$ provides a right lift with respect to $A \to B$. \qedhere
\end{enumerate}
\end{proof}

The following lemma is essentially \cite[Proposition 7.3]{Zeinalia16}. We provide a direct proof.
\begin{lemma}\label{lem-chains-equivalence}
Let $k$ be a field. The chain coalgebra functor $\Ch:\mathbf{qCat}^{0}\to \dgco$ preserves weak equivalences.
\end{lemma}
\begin{proof}
We reduce this lemma to three claims. 
\begin{enumerate}
\item $\Ch$ sends categorical equivalences between weak Kan complexes to weak equivalences.
\item $\Ch$ sends pushouts along disjoint unions of inner horn inclusions to trivial cofibrations.
\item There is a functor $\operatorname{Gx}^{\infty}$ sending each reduced simplicial set $A$ to a reduced weak Kan complex. For each reduced simplicial set $A$ there is a natural map $A \to \operatorname{Gx}^{\infty}A$ which is a colimit of pushouts along disjoint unions of inner horn inclusions.
\end{enumerate}

If we have these claims we may take any categorical equivalence $A \to B$ and using (3) replace it by a zig-zag $A \to \operatorname{Gx}^{\infty}A \to \operatorname{Gx}^{\infty}B \leftarrow B$. $\Ch$ sends the middle map to a weak equivalence by (1). The outer maps are sent to direct limits of trivial cofibrations, thus they are trivial cofibrations themselves, and $\Ch (A) \simeq \Ch (B)$.

To prove (1) it suffices to show that homotopy equivalences in $\qCat^{0}$ are sent to filtered quasi-isomorphisms. 
In fact we will show that homotopies of maps in $\qCat^{0}$ are sent to homotopies between maps of dg coalgebras.
 
Let $\mathbb I$ be a Kan complex such that the functor $X\mapsto X\times \mathbb I$ gives good cylinder objects in $\qCat$. For example, we can take for $\mathbb I$ the nerve of the category with two objects and two mutually inverse morphisms between them. We denote by ${\mathbb I}_{+}$ the simplicial set obtained by adding a disjoint base point. 

Then a cylinder object in $\qCat^{0}$ is given by the smash product $K \wedge {\mathbb I}_{+}$, i.e.\ $K\times {\mathbb I}_{+}/K \vee {\mathbb I}_{+}$. 
Thus any homotopy between two maps from $K$ to $K'$ in $\qCat^{0}$ may be represented by a map $F: K\wedge {\mathbb I}_{+} \to K'$.
This gives a map of coalgebras  $\Ch(F): \Ch(K\wedge \mathbb {\mathbb I}_{+}) \to \Ch K'$ and it suffices to show that $\Ch(K\wedge {\mathbb I}_{+})$ is a cylinder object in $\dgco_{k/}$.
For any coaugmented coalgebra $(C,w)$ we write $\tilde C$ for $C/w(k)$.
Then $\Ch(K)\amalg \Ch(K) \cong k \oplus \tilde \Ch(K) \oplus \tilde \Ch (K)$ injects into $\Ch(K\wedge {\mathbb I}_{+}) = k \oplus \tilde \Ch(K)\otimes \Ch(\mathbb I)$, thus it is a cofibration of dg coalgebras. 
It remains to show that $\Ch$ sends the projection to a filtered quasi-isomorphism.
Let $F^{0}(\Ch(K\wedge {\mathbb I}_{+})) = w(k)$ and $F^{i}(\Ch(K\wedge I_{+})) = F^{i}\tilde \Ch(K) \otimes \Ch(\mathbb I)$.
This is an admissible filtration and on graded pieces we have quasi-isomorphisms $\Gr^{i}\Ch(K)\otimes \Ch({\mathbb I}) \simeq \Gr^{i} \Ch(K)$.

To establish (2) we consider a simplicial set $K$ and let $K'$ be defined by attaching a collection of $n$-simplices $B_{i}$ along inner horns.
We need to show $\Ch(f): \Ch (K) \to \Ch (K')$ is a filtered quasi-isomorphism.
Filter $\Ch (K)$ by $F_{i}\Ch (K) = \oplus_{j \leq i} \Ch (K)_j$.
This is clearly an admissible filtration.
To define the filtration on $\Ch (K')$ we denote the face of $B_i$ that is not in $K$ by $b_{i}$. I.e.\ the $b_{i}$ are the  $(n-1)$-simplices which are in $K'$ but not in $K$.

We let 
$F'_i\Ch (K') = F_i\Ch (K)$ for $i <n$ and  $F'_i\Ch (K') = F_i\Ch (K) \oplus k.{B_i} \oplus k. {b}_i$ for $i \geq n$.
Thus every $n$-simplex appears in the the $n$-th graded piece of $K'$, with the exception of the $b_i$, which are in the $n$-th piece despite being $(n-1)$-simplices.

This is clearly compatible with differentials, we need to check the comultiplication. We check this on a basis.
By definition $\Delta B_i = \sum_k \partial_0^k B_i \otimes \partial_{max}^{n-k} B_i$.
Applying $\partial_0$ or $\partial_{max}$ $k$ times to $B_i$ gives a $n-k$ simplex which lives in $F'_{n-k}$ unless one of those terms is of the form $b_j$. 
For degree reasons this could only be $\partial_0 B_i$ and $\partial_{max} B_i$, but as we attached along inner horns both of these are in $K$, and thus in $F'_{n-1} \Ch (K')$.

Thus $F'$ gives an admissible filtration on $\Ch (K')$ which is clearly compatible with $\Ch(f)$.

$\Gr^{F}_{i} \Ch(K) \to \Gr^{F'}_{i}\Ch(K')$ is an isomorphism everywhere except for degree $n$. In  degree $n$ the cokernel has a basis give by all $B_i$ and $b_i$, and $dB_i = b_i \mod K$, so the cokernel is acyclic.

Thus $\Ch (K)$ and $\Ch (K')$ are filtered quasi-isomorphic. Since $\Ch (f)$ is a monomorphism it is a trivial cofibration. In this argument we fixed $n$ for ease of notation but the same argument goes through if we are attaching $n$-simplices for different values of $n$ simultaneously.

Claim (3) follows directly from the discussion after Definition 3.2.10 in \cite{Waldhausen13}.
The only change is that one defines $\operatorname{Gx}$ by filling all inner horns, rather than filling all horns.
\end{proof}

\begin{lemma}\label{lemma-cleftquillen}
Let $k$ be a field. Then the functor of the normalized chain coalgebra $\Ch: \mathbf{qCat}^{0} \to \dgco$ is left Quillen.
\end{lemma}
\begin{proof}
First we note that $\Ch$ has a right adjoint. It is provided by $C \mapsto \Hom_{\mathbf{dgCoa}}(\Ch(\Delta^{\bullet}), C)$ where $\Delta^{\bullet}$ is the cosimplicial simplicial set given by the $n$-simplex in degree $n$.

The fact that the adjunction is Quillen follows from Lemma \ref{lem-chains-equivalence} together with the observation that $\Ch$ preserves cofibrations, which are just monomorphisms in both categories. 
\end{proof}
\begin{remark}
	The reason for assuming that $k$ be a field in \ref{lem-chains-equivalence} and \ref{lemma-cleftquillen}	is that the category of dg coalgebras is only known to have a closed model category structure (with filtered quasi-isomorphisms as weak equivalences) under this assumption. Consequently, it is also needed for establishing dg Koszul duality as a Quillen equivalence between $\mathbf{dgA}_{/k}$ and $\dgco$ in \cite{Positselski11}.  This result should generalize to more general commutative rings, but there are technical difficulties in implementing it. We will establish Koszul duality as an equivalence of relative categories; this suffices for our purposes.
\end{remark}	
\begin{lemma}\label{lem:acyclic}
	Let $X$ be a complex of free $k$-modules such that for any field $F$ and a map $k\to F$ the complex $X\otimes_kF$ is acyclic. Then $X$ is acyclic to begin with.
\end{lemma}
\begin{proof}
It is well-known that a $k$-module is zero if and only if its localization at every maximal ideal of $k$ is zero; together with the exactness of the localization functor for modules over a commutative ring this implies that it suffices to assume that $k$ is local. Let its unique maximal ideal be generated by $x\in k$. Then we have the following homotopy pullback square, cf. \cite[Proposition 4.13]{DwG02}:
\[
\begin{CD}
X     @> >>  \hat{X}_{(x)}\\
@VVV        @VV V\\
X\otimes k[x^{-1}]    @> >> \hat{X}_{(x)}\otimes k[x^{-1}]
\end{CD}
\]
Here $X\to \hat{X}_{(x)}$ is the Bousfield localization of $X$ with respect to the functor $-\otimes k/(x)$ (it agrees with the completion of $X$ at the ideal $(x)\in k$).
Since $k/(x)$ and $k[x^{-1}]$ are both fields, we have that $X\otimes k[x^{-1}]$ and $\hat{X}_{(x)}$, and thus also $\hat{X}_{(x)}\otimes k[x^{-1}]$,  are acyclic and then so is $X$.
\end{proof}	
\begin{proposition}\label{prop-barcobar-z}
The relative categories $(\dga, W_{A})$ and $(\dgco, W_{C})$ are weakly equivalent; here $W_{A}$ denotes quasi-isomorphisms and $W_{C}$ weak equivalences of dg coalgebras.
\end{proposition}
\begin{proof}
We will prove that for any augmented dg algebra $A$ there is a quasi-isomorphism $\Om\Ba (A)\to A$ and for any conilpotent dg coalgebra $\Ch$ the natural map $C \to \Ba\Om (C)$ is a weak equivalence. If $k$ is a field this follows immediately from the results recalled in Section \ref{sect-barcobar}.

 Let $F$ be a field supplied with a map $k\to F$. Then by construction $\Ba \Om(A)\otimes F = \Ba \Om(A \otimes F)$. Thus $\Ba \Om(A \otimes F) \simeq A \otimes F$ implies $\Ba \Om(A)\otimes F \simeq A \otimes F$. But it follows from Lemma \ref{lem:acyclic} that two complexes of free $k$-modules are quasi-isomorphic if they are quasi-isomorphic after tensoring with any field; thus $\Om\Ba (A)\to A$ is a quasi-isomorphism.

The statement for dg coalgebras follows by applying the same argument to the graded pieces of the natural filtrations on $C$ and $\Ba \Om (C)$, see the proof of Theorem 6.10 in \cite{Positselski11}.

This shows that $\Om \Ba$ and $\Ba \Om$ are strictly homotopic to the identity functor on $\dga$ and $\dgco$ respectively in the sense of \cite{Barwick12}. So the two relative categories are strictly homotopy equivalent, and thus weakly equivalent by Proposition 7.5 (iii) in  \cite{Barwick12}.
\end{proof}

In the following formulation we denote by $W$, slighty abusing the notation, a submonoid of $M$, the corresponding subset of 1-simplices in $\Ne(M)$, and the corresponding subset of the canonical basis of $\Ch (M)$.

\begin{theorem}\label{thm-barandloop} 
Let $W \subset M$ be a submonoid.
Then there is a natural zig-zag of weak equivalences of  dg coalgebras $\Ch \Lo_{W}\Ne (M) \simeq \Ba\Lo_{W} \Ch (M)$.
\end{theorem}
\begin{proof}
By definition the localization constructions in dg algebras and simplicial sets are given by homotopy colimits, see Sections \ref{sect-alglocal} and \ref{sect-catlocal}.
As $\Ba$ is an equivalence of relative categories by Proposition \ref{prop-barcobar-z} it preserves homotopy colimits and we deduce 
$\Lo'_{W}\Ba \Ch (M) \simeq \Ba \Lo_W \Ch (M)$ where $$\Lo'_{W}\Ba \Ch (M) = \Ba \Ch (M) \coprod^{h}_{\amalg_{W} \Ba (k\langle t\rangle)}\amalg_{W}\Ba (k\langle t, t^{-1}\rangle)$$ where $\coprod^h$ stands for the homotopy pushout of dg coalgebras.

There is also a natural map $\eta: \Lo''_{W}\Ch \Ne (M)\to\Ch \Lo_{W}\Ne (M) $  where 
$$\Lo''_{W}\Ch \Ne (M) = \Ch \Ne (M) \coprod^{h}_{\amalg_{W}\Ch (I)}\amalg_{W}\Ch (J)$$
and $I,J$ are as in Section \ref{sect-catlocal}. We note first that $\eta$ is a weak equivalence if $k$ is a field since in that case $C$ is a left Quillen functor by Lemma \ref{lemma-cleftquillen} and so, it commutes with homotopy colimits. As the tensor product commutes with the homotopy colimit it follows that $\eta$ becomes a quasi-isomorphism after tensoring with an arbitrary field. Thus by Lemma \ref{lem:acyclic}  it is a weak equivalence in general.

It remains to identify the two different coalgebra localizations. We apply the isomorphic functors $\Ch \Ne $ and $\Ba \Ch $ to the map of discrete monoids $\set N \to \set Z$ to show that $\Ch (I) \to \Ch (J)$ is weakly equivalent to $\Ba (k\langle t \rangle) \to \Ba (k\langle t, t^{-1}\rangle)$.
\end{proof}

\section{Applications}
\subsection{The generalized correspondence of cobar and loop construction}
The second key ingredient for our applications is the following result of Fiedorowicz:
\begin{proposition}\label{prop-fiedorowicz}
There is a functor $\Mo: \sSet^{0} \to \Mon$ satisfying $K \simeq_{Q} \Ne \Mo(K)$.
\end{proposition}
\begin{proof}
By \cite[Theorem 3.5]{Fiedorowicz84} there is a a functor $\operatorname{D}$ from based path connected topological spaces to discrete monoids such that $X$ is weakly equivalent to the classifying space of $\operatorname{D}(X)$.
Then $\Mo(K):=\operatorname{D}(|K|)$.
\end{proof}

Applying Theorem \ref{thm-barandloop} in the case that $W = \Mo$ allows us to prove the following theorem that was proved for topological spaces in \cite{Zeinalia16}. It is a generalization of a classical result by Adams \cite{Adams56}.

To state the result we recall that the simplicial loop group $\G$ and the simplicial classifying space  $\overline W$ (constructed e.g.\ in \cite[Chapter V]{Goerss99}) give a Quillen equivalence between reduced simplicial sets and simplicial groups.

\begin{corollary}\label{cor-loopcobar}
Let $K$ be a grouplike reduced simplicial set. Then there is a natural zig-zag of quasi-isomorphism of dg algebras $\Ch \G  (K) \simeq \Omega \Ch (K)$.
\end{corollary}
\begin{proof} 
We denote a functorial fibrant replacement in the classical model structure by $R_{Q}$ and in the Joyal model structure by $R_{J}$. Then we note that $R_{J}K$ is weakly Kan and grouplike, thus it is a Kan fibrant replacement for $K$.  
By Proposition \ref{prop-fiedorowicz} we have $K \simeq_{Q} \Ne \Mo(K)$ and $R_{Q}K \simeq_{J} R_{Q}\Ne \Mo(K)$ as $\sSet$ is a Bousfield localisation of $\qCat$.
As $\Lo_{K_{1}}K \simeq_{J} K$ by assumption and $R_{J}\Lo_{K_{1}}$ is a Kan replacement (see Section \ref{sect-grouplike})
we obtain 
\[K \simeq_{J} R_{J}K \simeq_{J} R_{Q}K \simeq_{J} R_{J}\Lo_{\Mo(K)}\Ne \Mo(K) \simeq_{J} \Lo_{\Mo(K)}\Ne \Mo(K) .\] 

Thus $\Ch (K) \simeq \Ch \Lo_{\Mo(K)}\Ne \Mo(K) \simeq \Ba\Lo_{\Mo(K)}\Ch \Mo(K)$ by Theorem \ref{thm-barandloop} and $\Omega \Ch (K) \simeq \Lo_{\Mo(K)}\Ch \Mo(K)$ by Proposition \ref{prop-barcobar-z}.
Next $\Lo_{\Mo(K)}\Ch \Mo(K) \simeq \Ch \Lo_{\Mo(K)}\Mo(K)$ by \cite[Theorem 10.1]{BCL18}.
The classifying space of $\Lo_{\Mo(K)}\Mo(K)$ is weakly equivalent to $\Ne\Mo(K)$ by \cite[Theorem 5.5(ii)]{Dwyer80}.
As $\Lo_{\Mo(K)}\Mo(K)$ is a simplicial group (cf. \cite[Theorem 5.5(i)]{Dwyer80}), its classifying space may be computed by $\overline W$ \cite[Section V.4]{Goerss99}, so by the equivalence discussed above $\Lo_{\Mo(K)}\Mo(K)$ is of the form $\G\Ne\Mo(K)$, which is  weakly equivalent to $\G K$ by Proposition \ref{prop-fiedorowicz}. The result follows by applying $\Ch$.
\end{proof}

To go beyond grouplike simplicial sets we need to refine the loop group construction. The following almost trivial example is instructive.

\begin{example}
Consider the simplicial set $K$ with one $0$-simplex and one non-degenerate $1$-simplex. Topologically, $K$ is the circle, and so its loop space is the infinite cyclic group and the dg algebra $\Ch\G (K)$ is (quasi-isomorphic to) the ring of Laurent polynomials $k[t,t^{-1}]$ with $|t|=0$.

On the other hand, $\Om \Ch (K)\cong k[t]\neq k[t,t^{-1}]$. The reason for this discrepancy is that $K$ is \emph{not} grouplike.  
\end{example}
This example suggests that, even in the case when a simplicial set $K$ is not grouplike, the chains on its loop space could still be recovered as a localization of $\Om \Ch K$. This is indeed true:

\begin{corollary}\label{cor-hesstonks}
For any reduced simplicial set $K$ there is a weak equivalence $\Ch \G  (K) \simeq \Lo_{1+K_{1}}\Omega \Ch (K)$. Here the localization on the right hand side is performed at the set of cycles $\{1+s^{-1} x\}_{x \in K_{1}}$, where $s^{-1}$ denotes desuspension.
\end{corollary}
\begin{proof}
First we will show that $\Lo_{1+K_{1}}\Om \Ch (K) \simeq \Om \Ch \Lo _{K_{1}}(K)$ by commuting localization past $\Om$ and $\Ch$.

Since $\Om$ is an equivalence of relative categories it commutes with colimits. As in the proof of Theorem \ref{thm-barandloop} we may express the localization of a coalgebra as a homotopy pushout along $\amalg \ \Ch (I) \to \amalg \ \Ch (J)$ or equivalently along $\amalg \ \Ba (k\langle t \rangle) \to \amalg \ \Ba (k\langle t, t^{-1}\rangle)$. Again from the proof of Theorem \ref{thm-barandloop} we know that this localization commutes with $\Ch$.
Thus we have $\Om \Ch \Lo_{K_{1}}K \simeq \Om L_{K_{1}}\Ch(K) \simeq L_{1+K_{1}}\Om \Ch(K)$. Here for the last step we use that $\Om \Ba (k\langle t, t^{-1}\rangle) \simeq k\langle t, t^{-1}\rangle$. The equivalence from $\Ch(I)$ to $\Ba (k\langle t \rangle)$ sends an element $x \in K_{1}$ to $s^{-1}x-1$ in $\Om \Ch(K)$, cf. the correspondence in Lemma \ref{lemma-discretecommute}. Then $s^{-1}x-1$ is sent to $x$ by the natural transformation from $\Om B$ to the identity. Thus localizing $K$ at $K_{1}$ corresponds to localising $\Om \Ch K$ at $1 + K_{1}$. 

For the left hand side we note that $\Ch \G  (K) \simeq \Ch \G  \Lo_{K_{1}}(K)$ since $G$ preserves the (classical) weak equivalence between $K$ and $\Lo_{K_{1}}(K)$, and thus we deduce the result from Corollary \ref{cor-loopcobar} applied to $\Lo_{K_{1}}(K)$.
\end{proof}
\begin{remark}
This result throws some light on a construction of Hess and Tonks \cite{Hess10a}.
For a simplicial set $K$ that is not necessarily grouplike they consider an \emph{extended cobar construction} $\hat \Omega \Ch (K)$, see \cite[Section 1.2]{Hess10a}, and then show that $\Ch \G (K) \simeq \hat \Om \Ch (K)$ (in fact, they construct an explicit chain equivalence between these dg algebras).

Unravelling the extended cobar construction in the special case of a chain coalgebra we see that $\hat \Om \Ch (K)$ may be constructed as the dg algebra obtained from $\Om \Ch (K)$ by adding inverses for all the cycles $1+s^{-1}x$ for $x \in K_{1}$. As $\Om \Ch (K)$ is cofibrant over its subalgebra generated by these cycles, this is a derived localization, see \cite[Remark 3.11]{BCL18}.
Therefore we obtain that $\Ch \G  (K) \simeq \Lo_{1+K_{1}}\Omega \Ch (K) \simeq \hat \Om \Ch(K)$ by Corollary \ref{cor-hesstonks}, recovering the result of \cite{Hess10a}.

The construction of $\hat \Omega C$ for a dg coalgebra $C$ depends on the choice of a basis for $C_{1}$ and \cite{Hess10a} does not address the question whether different choices lead to quasi-isomorphic dg algebras. For $C=\Ch(K)$ there is a natural basis in $\Ch_1(K)$ given by $1$-simplices and with this basis the quasi-isomorphism $\Ch \G  (K) \simeq  \hat \Om \Ch(K)$ does hold. The following example shows that it will not hold with a wrong choice of basis.
\begin{example}
	Consider a monoid $M$ with two elements $1$ and $b$ where $1$ is the identity element and $b^2=b$. It is clear that $\Ne M$ is contractible and so $\Lo_{1+K_{1}}\Omega\Ch\Ne(M)\simeq \Ch(M)[b^{-1}]\cong k$. On the other hand, choosing the basis  in $(\Ch\Ne M)_1$ given by the \emph{negatives} of $1$-simplices in $\Ne M$ leads to the extended cobar construction $\hat \Om \Ch\Ne (M)$ that is quasi-isomorphic to $\Ch(M)[1+(1-b)]^{-1}\cong \Ch(M)[2-b]^{-1}$. It is easy to compute that $\Ch(M)[2-b]^{-1}\cong k[\frac{1}{2}]\times k$ and this is not isomorphic to $k$ unless $k$ has characteristic 2.
\end{example}
\end{remark}

\subsection{Chain coalgebras detect weak homotopy equivalences}
Next, we deduce the main result of  \cite{Rivera18} as follows:
\begin{corollary}\label{cor-detectwes}
Let $k = \set Z$. A map of reduced fibrant simplicial sets $f: K \to K' $ is a weak equivalence if and only if $f_{*}: \Ch (K)\to \Ch (K')$ is a weak equivalence of dg coalgebras.
\end{corollary}
\begin{proof}
The ``only if'' follows from Lemma \ref{lem-chains-equivalence} and Lemma \ref{lem:acyclic}.

To show the converse we assume that $f_{*}: \Ch (K) \simeq \Ch (K')$.
By Corollary \ref{cor-loopcobar} this implies that we have a quasi-isomorphism $\Ch \G (f): \Ch \G  (K) \simeq \Ch \G  (K')$.
Thus  $H_{0}(\Ch \G (f))$ is bijective. By construction it is a morphism of Hopf algebras, compatible with both the composition of loops and the coproduct.
Together this shows that $H_{0}(\Ch \G (f))$ induces an isomorphism between grouplike elements in $H_{0}(\G K)$ and $H_{0}(\G K')$, i.e.\ between the fundamental groups of $|K|$ and $|K'|$. 

We finish the proof by applying Whitehead's theorem. The identity components of $\G K$ and $\G K'$ are connected nilpotent spaces, thus by \cite{Dror71} they are weakly equivalent. As all components are equivalent and $f$ identifies the $\pi_{0}(\G K)$ and $\pi_{0}(\G K')$ we obtain a weak homotopy equivalence and thus a weak equivalence of simplicial monoids $\G K \to \G K'$. This implies $K \simeq_{Q} K'$.
\end{proof}

\subsection{Derived categories}
For the last part of this section we assume that $k$ is a field. We recall the derived categories of second kind constructed in \cite{Positselski11}. Specifically, for the coalgebra $\Ch (K)$ we consider the coderived category $\Dco(\Ch (K))$, which is a triangulated category obtained as the localization of the homotopy category of dg comodules over $\Ch (K)$ at morphisms with coacyclic cone. A dg comodule is coacyclic if it is contained in the minimal triangulated subcategory that contains the total complexes of short exact sequences and is closed under infinite direct sums.

A fundamental result says that for any conilpotent coalgebra $C$ there is an equivalence $\Dco(C) \simeq \De(\Om C)$, cf. \cite[Theorem 6.5(a)]{Positselski11}.
Thus weakly equivalent dg coalgebras have equivalent coderived categories. 

It follows directly from Lemma \ref{lem-chains-equivalence}  that the coderived category of the chain coalgebra of a simplicial set is an invariant with respect to Joyal weak equivalences. On the other hand, there is another homotopy invariant, this time with respect to classical (Quillen) weak equivalences of simplicial sets. It is the triangulated category of $\infty$-local systems on a simplicial set $K$. This could be defined e.g. as the derived category of cohomologically locally constant sheaves on  $|K|$, cf. \cite{Holstein1, HolsteinG}.

\begin{corollary}\label{cor-derivedcats}
The derived category of $\infty$-local systems on $K$ is a full subcategory of $\Dco(\Ch K)$. If $K$ is grouplike the two categories are equivalent.
\end{corollary}
\begin{proof}
By  \cite[Theorem 6.5(a)]{Positselski11} $\Dco(\Ch (K)) \simeq \De(\Om \Ch (K))$.
On the other hand, the derived category of $\infty$-local systems is $\De(\Ch \G (K))$, by the second part of Theorem 26 in \cite{Holstein1}.

If $K$ is grouplike, the two categories agree by Corollary \ref{cor-loopcobar}. Otherwise we have $\De(\Ch \G (K)) \simeq \De(\Lo_{1+K_{1}}\Om \Ch (K))$ by Corollary \ref{cor-hesstonks}, so $\infty$-local systems are modules over a localization of $\Om \Ch (K)$. But by Corollary 4.29 in \cite{BCL18} 
the derived category of modules over a localized dg algebra is a full subcategory of the derived category of modules over the original dg algebra. Explicitly, $\infty$-local systems are equivalent to the full subcategory of $K_{1}$-local objects in $\De(\Om \Ch (K))$.
\end{proof}

\section{An algebraic model for the homotopy category of spaces}
Finally, our results give us an algebraic model for the homotopy theory of connected topological spaces (equivalently, reduced simplicial sets). In this section we fix $k = \set Z$.

We consider the relative category $(\mathbf{Mon}, W)$ where $\mathbf{Mon}$ is the category of discrete monoids and $f: M \to N$ is in $W$ if and only if it induces a quasi-isomorphism of the derived localizations of the monoid algebras, i.e.\ if $\Lo_{M}\Ch (M) \simeq \Lo_{N}\Ch (N)$. 

This definition is completely algebraic in the sense that a monoid is an algebraic structure i.e. a set with a collection of finitary operations subject to finitely many identities \cite{Cohn81} and the notion of a weak equivalence in $\mathbf{Mon}$ is also described algebraically. 
The definition is meaningful because of the following:
\begin{corollary}\label{cor-monoidwe}
Two discrete monoids $B$ and $B'$ are weakly equivalent if and only if $\Ne (B) \simeq_{Q} \Ne (B')$.
\end{corollary}
\begin{proof}
By Theorem \ref{thm-barandloop} we know that $\Ba \Lo_{B}\Ch(B) \simeq \Ch \Lo_{B}\Ne(B)$ for any monoid, and together with Proposition \ref{prop-barcobar-z} this gives $\Lo_{B}\Ch (B) \simeq \Om \Ch \Lo_{B}\Ne (B)$.
This shows immediately that $\Ne (B) \simeq_{Q} \Ne (B')$ implies $B \simeq B'$.
Moreover $\Ne $ preserves weak equivalence as $\Om \Ch$ reflects weak equivalences by Corollary \ref{cor-detectwes}.
\end{proof}

\begin{theorem}\label{thm-main}
The nerve functor provides an equivalence of relative categories $\Ne: (\mathbf{Mon}, W) \to (\sSet^{0}, W_{Q})$.
\end{theorem}
\begin{proof}
$N$ preserves weak equivalences by Corollary \ref{cor-monoidwe}. 
Using the functor $\Mo$ from \ref{prop-fiedorowicz} we have $K \simeq \Ne \Mo(K)$.

Moreover for any monoid $B$ to show $B \simeq \Mo \Ne(B)$ it suffices to show that $\Ne (B) \simeq \Ne \Mo \Ne(B)$, which follows immediately from the above.

This shows that $(\mathbf{Mon}, W)$ and $(\sSet^{0}, W_{Q})$ are homotopy equivalent and  thus weakly equivalent, cf.\ the proof of Proposition \ref{prop-barcobar-z}.
\end{proof}

\bibliography{./biblibrary2}

\begin{thebibliography}{10}

\bibitem{Adams56}
{\sc J.~Adams}, {\em {On the cobar construction}}, Proc. Natl. Acad. Sci.
  U.S.A., 42 (1956), pp.~409--412.

\bibitem{Barnea17}
{\sc I.~Barnea, Y.~Harpaz, and G.~Horel}, {\em Pro-categories in homotopy
  theory}, Algebraic \& Geometric Topology, 17 (2017), pp.~567--643.

\bibitem{BaK12}
{\sc C.~Barwick and D.~M. Kan}, {\em A characterization of simplicial
  localization functors and a discussion of {DK} equivalences}, Indag. Math.
  (N.S.), 23 (2012), pp.~69--79.

\bibitem{Barwick12}
{\sc C.~Barwick and D.~M. Kan}, {\em Relative categories: another model for the
  homotopy theory of homotopy theories}, Indagationes Mathematicae, 23 (2012),
  pp.~42--68.

\bibitem{BCL18}
{\sc C.~Braun, J.~Chuang, and A.~Lazarev}, {\em Derived localisation of
  algebras and modules}, Adv. Math., 328 (2018), pp.~555--622.

\bibitem{HolsteinG}
{\sc J.~{Chuang}, J.~{Holstein}, and A.~{Lazarev}}, {\em {Maurer-Cartan moduli
  and theorems of Riemann-Hilbert type}}, Appl. Categor. Struct.
  https://doi.org/10.1007/s10485-021-09631-3,  (2021),
  \href{http://arxiv.org/abs/1802.02549}{{ arXiv:1802.02549}}.

\bibitem{Cisinski?}
{\sc D.-C. Cisinski}, {\em Higher categories and homotopical algebra}.
\newblock Lectures notes availble at
  {www.mathematik.uni-regensburg.de/cisinski/CatLR.pdf}.

\bibitem{Cohn81}
{\sc P.~M. Cohn}, {\em Universal algebra}, vol.~6 of Mathematics and its
  Applications, D. Reidel Publishing Co., Dordrecht-Boston, Mass., second~ed.,
  1981.

\bibitem{Dror71}
{\sc E.~Dror}, {\em A generalization of the {W}hitehead theorem},  (1971),
  pp.~13--22. Lecture Notes in Math., Vol. 249.

\bibitem{DwG02}
{\sc W.~G. Dwyer and J.~P.~C. Greenlees}, {\em Complete modules and torsion
  modules}, Amer. J. Math., 124 (2002), pp.~199--220.

\bibitem{Dwyer80}
{\sc W.~G. {Dwyer} and D.~M. Kan}, {\em {Simplicial localizations of
  categories}}, J. Pure Appl. Algebra, 17 (1980), pp.~267--284.

\bibitem{Fiedorowicz84}
{\sc Z.~Fiedorowicz}, {\em Classifying spaces of topological monoids and
  categories}, American Journal of Mathematics, 106 (1984), pp.~301--350.

\bibitem{Goerss99}
{\sc P.~Goerss and J.~Jardine}, {\em Simplicial homotopy theory}, Progress in
  Mathematics Series, Birkh\"auser Verlag, 1999.

\bibitem{Hess10a}
{\sc K.~Hess and A.~Tonks}, {\em The loop group and the cobar construction},
  Proceedings of the American Mathematical Society, 138 (2010), pp.~1861--1876.

\bibitem{Holstein1}
{\sc J.~V.~S. {Holstein}}, {\em Morita cohomology}, Math. Proc. Camb. Phil.
  Soc., 158 (2015), pp.~1--26.

\bibitem{Husemoller74}
{\sc D.~Husemoller, J.~C. Moore, and J.~Stasheff}, {\em Differential
  homological algebra and homogeneous spaces}, Journal of Pure and Applied
  Algebra, 5 (1974), pp.~113--185.

\bibitem{Joyal?}
{\sc A.~Joyal}, {\em The theory of quasi-categories and its applications}.
\newblock Notes available at
  {http:mat.uab.cat/~kock/crm/hocat/advanced-course/Quadern45-2.pdf}.

\bibitem{Lurie11a}
{\sc J.~Lurie}, {\em {Higher Topos Theory}}, Annals of Mathematics Studies, 170
  (2011).

\bibitem{Neisendorfer10}
{\sc J.~Neisendorfer}, {\em Algebraic methods in unstable homotopy theory},
  vol.~12, Cambridge University Press, 2010.

\bibitem{Positselski11}
{\sc L.~Positselski}, {\em Two kinds of derived categories, {K}oszul duality,
  and comodule-contramodule correspondence}, Mem. Amer. Math. Soc., 212 (2011),
  pp.~vi+133.

\bibitem{Rap10}
{\sc G.~Raptis}, {\em Homotopy theory of posets}, Homology Homotopy Appl., 12
  (2010), pp.~211--230.

\bibitem{Rivera18}
{\sc M.~Rivera, F.~Wierstra, and M.~Zeinalian}, {\em The functor of singular
  chains detects weak homotopy equivalences}, Proc. Amer. Math. Soc., 147
  (2019), pp.~4987--4998.

\bibitem{Zeinalia16}
{\sc M.~Rivera and M.~Zeinalian}, {\em Cubical rigidification, the cobar
  construction and the based loop space}, Algebr. Geom. Topol., 18 (2018),
  pp.~3789--3820.

\bibitem{Tho80}
{\sc R.~W. Thomason}, {\em Cat as a closed model category}, Cahiers Topologie
  G\'{e}om. Diff\'{e}rentielle, 21 (1980), pp.~305--324.

\bibitem{Waldhausen13}
{\sc F.~Waldhausen, B.~Jahren, and J.~Rognes}, {\em Spaces of PL Manifolds and
  Categories of Simple Maps (AM-186)}, Annals of Mathematics Studies, Princeton
  University Press, 2013.

\end{thebibliography}

\end{document}